\documentclass{amsart}
\usepackage{mathrsfs}
\usepackage{amsfonts}
\usepackage{amsfonts,amssymb,amsmath,amsthm}
\usepackage{url}
\usepackage{enumerate}

\theoremstyle{plain}
\newtheorem{theorem}{Theorem}[section]
\newtheorem{thm}{Theorem}[section]
\newtheorem{lemma}[thm]{Lemma}
\newtheorem{corollary}[thm]{Corollary}
\newtheorem{corollary*}[thm]{Corollary*}
\newtheorem{proposition}[thm]{Proposition}
\newtheorem{proposition*}[thm]{Proposition*}

\newtheorem{open problem}[thm]{Open Problem}
\newtheorem{remark}[thm]{Remark}
\theoremstyle{definition}

\newtheorem{defn-thm}[thm]{Definition-Theorem}
\numberwithin{equation}{section}

\title[Surfaces pinched by normal curvature for mean curvature flow ]
{ Surfaces pinched by normal curvature for mean curvature flow in space forms }

\author{Dong Pu}
\address{Center of Mathematical Sciences, Zhejiang University,
             Hangzhou,  310027, People's Republic of China}
\email{pudong@zju.edu.cn}

\author{Jingjing Su}

\address{
Center of Mathematical Sciences, Zhejiang University,
             Hangzhou,  310027, People's Republic of China}

\email{sujj@zju.edu.cn}

\author{Hongwei Xu}

\address{Center of Mathematical Sciences, Zhejiang University,
             Hangzhou,  310027, People's Republic of China}
\email{xuhw@zju.edu.cn}

\thanks{Research supported by the National Natural Science Foundation of China, Grant Nos. 11531012, 11371315.}



\begin{document}


\begin{abstract} In this paper, we investigate the mean curvature flow of compact surfaces in $4$-dimensional space forms. We prove the convergence theorems for the mean curvature flow under certain pinching conditions involving the normal curvature, which generalise Baker-Nguyen's convergence theorem.
\end{abstract}
 \maketitle

\section{Introduction}
Let $M$ be an $n$-dimensional compact submanifold isometrically
immersed in a Riemannian manifold $N^{n+d}$. Denoted by $F_0$ the
isometric immersion. We consider the one-parameter family
$F_{t}=F(\cdot,t)$ of immersions $F_{t}:M\rightarrow N^{n+d}$ with
corresponding images $M_{t}=F_{t}(M)$ which satisfies
\begin{equation}
  \left\{ \begin{array}{l}
\frac{\partial}{\partial t}F(x,t)=H(x,t),\\
F(x,0)=F_0(x),
\end{array} \right.
\end{equation}
where $H(x,t)$ is the mean curvature vector of $M_t$. We call
$F_{t}:M\rightarrow N^{n+d}$ the mean curvature flow with initial
value $F_0$. In 1980's,  Huisken proved the convergence theorems of
the mean curvature flow of hypersurfaces under certain conditions in
a series of papers.
 For the initial hypersurfaces satisfying the convexity condition, Huisken \cite{H1,H2} proved that the
solution of the mean curvature flow converges to a point as the time
approaches the finite maximal time. Motivated by the rigidity theorem for hypersurfaces with constant mean curvature in spheres due to Okumura \cite{Ok}, Huisken \cite{H3}
proved that the mean curvature flow of  hypersurfaces in the unit
sphere $\mathbb{S}^{n+1}$ under a pointwise curvature pinching condition
either converges to a round point in finite time, or
converges to a total geodesic sphere of $\mathbb{S}^{n+1}$ as
$t\rightarrow\infty$.

During the past three decades, there are some important progresses
on the theory of mean curvature flows of hypersurfaces. On the other
hand, a few striking results for mean curvature flows of higher
codimension were obtained by several geometers. For instance,
Chen-Li \cite{CL} and Wang \cite{W1} studied the mean curvature flow
of surfaces in a four dimensional K$\ddot{\rm{a}}$hler-Einstein
manifold. For the mean curvature flow of arbitrary codimension in
the Euclidean space, Andrews and Baker \cite{AB} considered the
convergence problem under the curvature pinching condition for
$|H|^2$ and the squared norm of the second fundamental form $|A|^2$.
They proved the following convergence theorem.
\begin{theorem}\label{AB}
 Let $F_0 :M^{n} \rightarrow \mathbb{R}^{n+d}$ be a smooth compact
submanifold with $|H|_{min}>0$ in a Euclidean space.
Assume $M$ satisfies
\begin{eqnarray}
\label{pinch-cond}
|A|^2\leq\begin{cases}
            \frac{4}{3n}|H|^2, \ &n = 2, 3, \\
            \frac{1}{n-1}|H|^2, \ &n \geq 4.
        \end{cases}
\end{eqnarray}
Then the mean curvature flow with the initial value $F_0$ converges to
a round point in finite time.
\end{theorem}

 In \cite{B}, Baker investigated the mean curvature flow of arbitrary codimension in spheres and obtained the following result.
\begin{theorem}\label{AB2}
Let $F_0 :M^{n} \rightarrow \mathbb{S}^{n+d}(\frac{1}{\sqrt{\bar K}})$ be a smooth compact
submanifold in a sphere with constant curvature $\bar K$.
Assume $M$ satisfies
\begin{equation}
|A|^2\leq\begin{cases}
            \frac{4}{3n}|H|^2+\frac{2(n-1)}{3}\bar K, \ &n = 2,3, \\
            \frac{1}{n-1}|H|^2+2\bar K, \ &n \geq 4.
        \end{cases}
\end{equation}
Then
  the mean curvature flow with the initial value $F_0$
either converges to a round point in finite time, or
converges to a total geodesic sphere of $\mathbb{S}^{n+d}(\frac{1}{\sqrt{\bar K}})$ as
$t\rightarrow\infty$.
\end{theorem}

Afterwards, Lei-Xu \cite{LX2} obtained a sharp convergence theorem for mean curvature flow of higher codimension in spheres. Set
\begin{equation}
a(x) = n \bar K+ \frac{n}{2 ( n-1 )} x -
\frac{n-2}{2
   ( n-1 )} \sqrt{x^{2} +4 ( n-1 ) \bar K x},
\end{equation} they proved the following theorem.
\begin{theorem}\label{LX2}
Let $F_0 :M^{n} \rightarrow \mathbb{S}^{n+d}(\frac{1}{\sqrt{\bar K}})$ be an n-dimensional $(n\geq6)$ smooth compact
submanifold in a sphere with constant curvature $\bar K$.
Assume $M$ satisfies
\begin{equation}
|A|^2< c(n,|H|,\bar K).
\end{equation}
Then
  the mean curvature flow with the initial value $F_0$
either converges to a round point in finite time, or
converges to a total geodesic sphere of $\mathbb{S}^{n+d}(\frac{1}{\sqrt{\bar K}})$ as
$t\rightarrow\infty$.

Here $c(n,|H|,\bar K)$ is an explicit positive scalar defined by
$$c(n,|H|,\bar K)=\min\{a(|H|^2), b(|H|^2) \},$$
where
\begin{eqnarray*}
  b(x) &=&  a(x_0)+ a'(x_0)(x-x_0)+\frac12a''(x_0)(x-x_0)^2, \\
  x_0 &=& \frac{2n+2}{n-4}\sqrt{n-1}(\sqrt{n-1}-\frac{n-4}{2n+2})^2\bar K.
\end{eqnarray*}
\end{theorem}
The scalar $c(n,|H|,\bar K)$ in Theorem \ref{LX2} satisfies the following: (i) $c(n,|H|,\bar K)>\frac{1}{n-1}|H|^2+2\bar K$; (ii) $c(n,|H|,\bar K)>\frac76\sqrt{n-1}\bar K$; (iii) $c(n,|H|,\bar K)=a(|H|^2)$ when $|H|^2\geq x_0$.
For mean curvature flow of surfaces in hyperbolic spaces,  Liu-Xu-Ye-Zhao \cite{LXYZ} proved the following convergence theorem.
\begin{theorem}\label{A}
Let $F_0: M^{n}\rightarrow \mathbb{H}^{n+d}(\bar K)$ be a smooth compact
submanifold in a hyperbolic space with constant curvature $\bar K$.
Assume $M$ satisfies
\begin{eqnarray}
\label{pinch-cond}
|A|^2\leq\begin{cases}
            \frac{4}{3n}|H|^2+\frac{n}{2}\bar K, \ &n = 2, 3, \\
            \frac{1}{n-1}|H|^2+2\bar K, \ &n \geq 4.
        \end{cases}
\end{eqnarray}
Then the mean curvature flow with the initial value $F_0$ converges to
a round point in finite time.
\end{theorem}

Recently, Lei-Xu \cite{LX} proved a convergence theorem of arbitrary codimension in hyperbolic spaces that the initial submanifold $M^n$ of
dimension $n(\ge6)$ under the optimal pinching condition.
\begin{theorem}\label{LX1}
Let $F_0: M^{n}\rightarrow \mathbb{H}^{n+d}(\bar K)$ be an n-dimensional $(n\geq6)$ smooth complete
submanifold in a hyperbolic space with constant curvature $\bar K$.
Assume $M$ satisfies
\begin{eqnarray}
\underset{M}{\sup}\big(|A|^2- a(|H|^2)\big)<0, \ where \ |H|^2+n^2 \bar K>0.
\end{eqnarray}
Then the mean curvature flow with the initial value $F_0$ converges to
a round point in finite time.
\end{theorem}

Since $a(|H|^2)>\frac{1}{n-1}|H|^2+2\bar K$, the theorem above improves Theorem
\ref{A} for $n \geq 6$. The example in \cite{LX} shows that the
pinching condition in Theorem \ref{LX1} is optimal for arbitrary $n(\geq 6)$.  Note that initial submanifolds in the almost all convergence
results possess positive curvatures. The convergence theorems in Theorem \ref{LX2} and Theorem \ref{LX1} imply
that the Ricci curvatures of the initial submanifolds are positive, but don't imply
the positivity of the sectional curvatures. Therefore, the two
convergence theorems also imply the differentiable sphere theorems
for submanifolds with positive Ricci curvatures.

However, there is no optimal pinching condition for surfaces in space forms under mean curvature flow. Recently, Baker and Nguyen \cite{BN} introduced a new pinching condition evolving the normal curvature $K^\perp$ for surfaces in $\mathbb{R}^4$, and obtained
\begin{theorem}\label{C}
Let $F_0 :M \rightarrow \mathbb{R}^{4}$ be a
  compact surface with $|H|_{min}>0$. Assume $M$ satisfies
\begin{equation}|A|^2+2(1-\frac{4}{3}k)|K^\perp|\leq k|H|^2,
\end{equation}
where $ k\leq \frac{29}{40}$. Then
  the mean curvature flow with the initial value $F_0$ converges to a round point in finite time.
\end{theorem}

Theorem \ref{C} improves the coefficient of $|H|^2$ in Theorem \ref{AB} for codimension two surfaces from $\frac23$ to $\frac{29}{40}$ with the help of the normal curvature. Baker and Nguyen conjectured that the Clifford torus viewed as a surface of codimension two in $\mathbb{R}^{4}$ is the true obstruction to the theorem \ref{C}, corresponding to the optimal constant $k=1$. The mean curvature flow in complex projective spaces was investigated in \cite{LX3,P}. The blowup analysis of the mean curvature flow at singular time was investigated in \cite{CH,LN,LW,XYZ}. For other results and applications of the mean curvature flow, we refer the readers to \cite{H4,LX4,LXYZ1,LXZ1,M,Sm,W2,Z}.

Motivated by above theorems, we consider surfaces pinched by the normal curvature for the mean curvature flow in the simply connected space form $\mathbb{F}^{4}(\bar{K})$ with constant curvature $\bar K$. Putting
\begin{equation}\label{gamma}
  \gamma=\begin{cases}
            1-\frac{3}{2}k, \ &\frac12<k\leq \frac{2}{3}, \\
            1-\frac{4}{3}k, \ &\frac23<k\leq \frac{29}{40},
        \end{cases}
\end{equation}
we prove the following convergence theorems.

\begin{thm}\label{main2}
Let $F_0 :M \rightarrow \mathbb{S}^{4}(\frac{1}{\sqrt{\bar K}})$ be a compact surface immersed in the sphere with
constant curvature $\bar K$. If $M$ satisfies
\begin{equation}\label{cons}|A|^2+2\gamma|K^\perp| \leq k|H|^2+\beta \bar K, \ \ \beta= 4k-2,
\end{equation}
then the mean curvature flow with the initial value $F_0$  converges to a round point in finite time, or
converges to a total geodesic sphere of $\mathbb{S}^{4}(\frac{1}{\sqrt{\bar K}})$ as
$t\rightarrow\infty$.
\end{thm}

In particular if $k=\frac23$, Theorem \ref{main2} is the same as the case of $n=2$ and $d=2$ in Theorem \ref{AB2}. Consider the compact surface  $M=\mathbb{S}^1(r)\times \mathbb{S}^1(s)\subset \mathbb{S}^3(1)\subset \mathbb{S}^4(1)$ with $r^2+s^2=1$. We have $|A|^2 = |H|^2+2$, which implies that it doesn't satisfy the pinching condition.

\begin{thm}\label{main}
Let $F_0 :M \rightarrow \mathbb{H}^{4}(\bar K)$ be a compact surface immersed in the hyperbolic space with
constant curvature $\bar K$. If $M$ satisfies
\begin{equation}\label{conh}|A|^2+2\gamma|K^\perp| \leq k|H|^2+\beta \bar K, \ \ \beta= 4-\frac{2}{k},
\end{equation}
then the mean curvature flow with the initial value $F_0$  converges to a round point in finite time.
\end{thm}

In particular if $k=\frac23$, Theorem \ref{main} is the same as the case where $n=2$ and $d=2$ in Theorem \ref{A}. Consider the compact surface  $M=\mathbb{S}^1(r)\times \mathbb{S}^1(s)\subset \mathbb{S}^3(1)\subset \mathbb{H}^4(-1)$ with $r^2+s^2=1$. We have $|A|^2 = |H|^2-2$ and $|H|^2\geq8$, which implies that it doesn't satisfy the pinching condition.

In 1999, De Smet et al. \cite{Smet} proposed the well-known DDVV conjecture which was proved by Ge-Tang \cite{Ge} and Lu \cite{Lu}. Applying the DDVV inequality for codimension two surfaces in the space form, i.e., $2|K^\perp|\leq |A|^2-\frac{|H|^2}{2}$, and taking $k=\frac{29}{40}$ we have the following corollary.
\begin{corollary}\label{cor}Let $F_0 :M \rightarrow \mathbb{F}^{4}(\bar K)$ be a compact surface immersed in the space form with $|H|^2+4\bar K>0 $. Assume that $M$ satisfies
\begin{equation}\label{cor2}|A|^2 \leq \frac{89}{124}|H|^2+\beta \bar K,
\end{equation}
where $\beta=\begin{cases}
            \frac{27}{31}, \ & \bar{K}\geq0, \\
            \frac{1080}{899}, \ & \bar{K}<0,
          \end{cases}$
then the mean curvature flow with the initial value $F_0$  converges to a round point in finite time, or
converges to a total geodesic sphere of $\mathbb{S}^{4}(\frac{1}{\sqrt{\bar K}})$ as
$t\rightarrow\infty$.
\end{corollary}
For $\bar{K}>0$, $|H|^2+4\bar K>0$ is automatically satisfied. For  $\bar{K}=0$
, $|H|^2+4\bar K>0$ is equivalent to that the mean curvature is nowhere vanishing.
For $\bar{K}<0$, $|H|^2+4\bar K>0$ is implied by condition (\ref{cor2}). Corollary \ref{cor} improves Theorem \ref{AB}, Theorem \ref{AB2} and Theorem \ref{A} for codimension two surfaces. In fact, for $\bar{K}<0$ we have $|A|^2\leq \frac{2}{3}|H|^2+\bar K<\frac{89}{124}|H|^2+\frac{1080}{899} \bar K$ with $|H|^2+4\bar K>0$.

Based on the convergence results and the examples above, we propose the following optimal problem.
\begin{open problem}Let $F_0 :M \rightarrow \mathbb{F}^{4}(\bar K)$ be a compact surface immersed in the space form with $|H|^2+4\bar K>0 $. Assume that $M$ satisfies
\begin{equation}|A|^2+2(1-k)|K^\perp| < k|H|^2+\beta \bar K,
\end{equation}
for $\frac12 < k \leq 1$ and $\beta=\begin{cases}
            4k-2, \ & \bar{K}\geq0, \\
            4-\frac{2}{k}, \ & \bar{K}<0.
          \end{cases}$ Is it possible to prove that
  the mean curvature flow with the initial value $F_0$  converges to a round point in finite time, or
converges to a total geodesic sphere of $\mathbb{S}^{4}(\frac{1}{\sqrt{\bar K}})$ as
$t\rightarrow\infty$?
\end{open problem}

$  \ $
\section{Preliminaries}

Let $( M^{n} ,g )$ be the $n$-dimensional Riemannian submanifold isometrically immersed in a simply connected space form
$\mathbb{F}^{n+d} ( \bar K)$ with constant curvature $\bar K$. Denote by
$\bar{\nabla}$ the Levi-Civita connection of the ambient space
$\mathbb{F}^{n+d} $. We use the same symbol $\nabla$ to represent the
connection of the tangent bundle $T M$ and the normal bundle $N M$. Denote by
$( \cdot )^{\top}$ and $( \cdot )^{\bot}$ the projections onto $T M$ and $N
M$, respectively. For $u,v \in \Gamma ( T M )$, $\xi \in \Gamma ( N M )$, the
connection $\nabla$ is given by $\nabla_{u} v= ( \bar{\nabla}_{u} v )^{\top}$
and $\nabla_{u} \xi = ( \bar{\nabla}_{u} \xi )^{\perp}$. The second fundamental
form of $M$ is defined by
\[ A ( u,v ) = ( \bar{\nabla}_{u} v )^{\perp} . \]

Let $\{ e_{i}   \,|\,  1 \le i \le n \}$ be a local orthonormal frame
for the tangent bundle and $\{ \nu_{\alpha}   \,|\,  n+1 \le \alpha \le
n+d \}$ be a local orthonormal frame for the normal bundle. Let $\{ \omega_{i}
\}$ be the dual frame of $\{ e_{i} \}$. With the local frame, the first and
the second fundamental forms can be written as $g= \sum_{i} \omega^{i} \otimes
\omega^{i}$ and $A= \sum_{i,j, \alpha} h_{i j \alpha}   \omega^{i} \otimes
\omega^{j} \otimes \nu_{\alpha}=\sum_{i,j} h_{i j }   \omega^{i} \otimes
\omega^{j}$, respectively. The mean curvature vector is
given by
\[ H= \sum_{\alpha} H_{\alpha} \nu_{\alpha} , \hspace{1em} H_{\alpha} =
   \sum_{i} h_{i i \alpha} . \]
Let
$\mathring{A} =A - \tfrac{H}{n}  g$ be the
traceless second fundamental form, whose squared norm satisfies
$|\mathring{A}|^{2} = | A |^{2} - \frac{|H|^{2}}{n} $.
Denote by $\nabla^{2}_{i,j} T= \nabla_{i} ( \nabla_{j} T ) -
\nabla_{\nabla_{i} e_{j}} T$ the second order covariant derivative of tensors.
Then the Laplacian of a tensor is defined by $\Delta T= \sum_{i}
\nabla^{2}_{i,i} T$.

The normal curvature tensor in local orthonormal frames for the tangent and normal bundles is given by
\begin{equation}
  {Rm}^{\perp}=R_{ij\alpha\beta}^\bot=h_{ip\alpha}h_{jp\beta}-h_{jp\alpha}h_{ip\beta}.
\end{equation}

We have the following evolution equations for the mean curvature flow.
\begin{lemma}[\cite{AB}]\label{ee}
\begin{eqnarray}
  \nabla_{\partial_t} h_{ij} &=& \Delta h_{ij} +h_{ij}\cdot h_{pq}h_{pq}+h_{iq}\cdot h_{qp}h_{pj}+h_{jq}\cdot h_{qp}h_{pi} \label{ee1} \\
   && -2h_{iq}\cdot h_{jp}h_{pq}+2\bar KHg_{ij}-n\bar Kh_{ij}, \nonumber \\
  \nabla_{\partial_t} H &=& \Delta H +H\cdot h_{pq} h_{pq}+n\bar K H, \\
  \frac{\partial}{\partial t}|A|^2 &=& \Delta|A|^2-2|\nabla A|^2+2R_1+4\bar K|H|^2-2n\bar K |A|^2, \\
  \frac{\partial}{\partial t}|H|^2 &=& \Delta|H|^2-2|\nabla H|^2+2R_2+2n\bar K|H|^2,
\end{eqnarray}

where
\begin{eqnarray}
R_1&=&2\sum\limits_{\alpha,\beta}(\sum\limits_{i,j}h_{ij\alpha}h_{ij\beta})^2+|{Rm}^{\perp}|^2,\\
R_2&=&\sum\limits_{i,j}(\sum\limits_{\alpha}H_\alpha h_{ij\alpha})^2.
\end{eqnarray}
\end{lemma}

From now we consider the mean curvature flow of surfaces in $\mathbb{F}^{4}( \bar K)$. Let $\{ e_{i} \}$ be a frame diagonalizing the matrix $(h_{ij3})$. If $|H|\neq 0$, we choose $\nu_3=\frac{H}{|H|}$, then
$tr (h_{ij3})=|H|$ and $tr (h_{ij4})=0$. Set
\begin{equation}
  A=A_1+A_2=\begin{bmatrix} |H|/2+a & 0 \\ 0 & |H|/2-a \end{bmatrix} \nu_3 +\begin{bmatrix} b & c \\ c & -b\end{bmatrix} \nu_4,
\end{equation}
we have
\begin{equation}
  |\mathring{A}|^2=|\mathring{A}_1|^2+|\mathring{A}_2|^2=2a^{2}+2b^{2}+2c^{2},
\end{equation}
where $\mathring{A}_1=A_1-\frac{H}{2}g$ and $\mathring{A}_2=A_2$.

For surfaces of codimension two, we denote the normal curvature $R^\perp _{1234}$ by $K^\perp$ and we have
\begin{align}
  K^\perp=2ac&\leq a^{2}+b^{2}+c^{2}=\frac12 |\mathring{A}|^2,\label{Ka} \\
 |\nabla  K^\perp|&\leq 4 |\mathring{A}||\nabla \mathring{A}|.\label{Kac}
\end{align}

The following lemma is the evolution equation of the length of the normal curvature.
\begin{lemma}\label{eeK}
\begin{align}\label{norm}
\frac{\partial  }{ \partial t } |K^{\perp}|= \Delta |K^{\perp}| - 2 \frac{ K^{\perp} }{|K^{\perp}|  }\nabla_{\!\!evol}K^{\perp} + R_3-4\bar K|K^{\perp}|,
\end{align}
where
\begin{eqnarray}
  \nabla_{evol}K^\bot &=& \sum\limits_{p,q}(\nabla_qh_{1p3}\nabla_qh_{2p4}
-\nabla_qh_{2p3}\nabla_qh_{1p4}), \\
  R_3 &=& |K^{\perp}|\left( |A|^2 + 2|\mathring{A}|^2 - 2b^2 \right).
\end{eqnarray}
\end{lemma}
\begin{proof}
From the definition of $R_{ij\alpha\beta}^\bot$ and Lemma \ref{ee}, we have
\begin{align*}
\frac{\partial}{\partial t} R ^ { \perp} _{ ij\alpha \beta } &= \Delta R ^ { \perp} _{ ij \alpha \beta } - 2 \sum _{p,q }\left ( \nabla_q h _{ ip\alpha} \nabla _ q h _{ jp\beta} - \nabla _ q h _{ jp \alpha } \nabla _ q h _{ i p \beta }\right ) \\
 &\quad +\sum_{ p } \left (\frac{ d}{ dt} h _{ i p \alpha} h _{ j p \beta } + h _{ i p \alpha} \frac{ d } {dt } h _{ j p \beta } - \frac { d} { dt } h _{ j p \alpha } h _ {i p \beta } - h _{ jp \alpha} \frac{ d}{ dt }h_{ i p \beta }\right),
\end{align*}
where $\frac{ d } {dt } h _{ ij\alpha }$ is the reaction terms of $\nabla_{\partial_t} h_{ij}$, i.e.,
\begin{eqnarray}
  \frac{ d } {dt } h _{ ij\alpha } &=& h_{ij}\cdot h_{pq}h_{pq\alpha }+h_{iq}\cdot h_{qp}h_{pj\alpha }+h_{jq}\cdot h_{qp}h_{pi\alpha } \\
  && -2h_{iq}\cdot h_{jp}h_{pq\alpha }+2\bar Kg_{ij}H_\alpha -n\bar Kh_{ij\alpha}.\nonumber
\end{eqnarray}
Since $R^\perp_{1234}=K ^ \perp$, we denote  by $\nabla_{\!\!evol}K^{\perp}$ the gradient terms $$\sum\limits_{p,q}(\nabla_qh_{1p3}\nabla_qh_{2p4}
-\nabla_qh_{2p3}\nabla_qh_{1p4}),$$
and we have
\begin{equation}
  \frac{\partial}{\partial t}K ^ \perp= \Delta K^{\perp}-2\nabla_{\!\!evol}K^{\perp}+  K^{\perp}\left( |A|^2 + 2|\mathring{A}|^2 - 2b^2 -4\bar K\right).
\end{equation}
Thus from the evolution equation of $K ^ \perp$, we get
\begin{eqnarray}\label{norm}
\frac{\partial  }{ \partial t } |K^{\perp}|&=& \Delta |K^{\perp}| - 2 \frac{ K^{\perp} }{|K^{\perp}|  }\nabla_{\!\!evol}K^{\perp}\\
 &&+ |K^{\perp}|\left( |A|^2 + 2|\mathring{A}|^2 - 2b^2\right) -4\bar K|K^{\perp}|.\nonumber
\end{eqnarray}
\end{proof}

We have the following following estimates for the gradient of the second fundamental form.
\begin{lemma}[\cite{BN} Proposition 4.2]
\begin{align}
|\nabla A | ^ 2 &\geq \frac  {3 }{ 4 } | \nabla H | ^ 2,  \label{eqn_gradient1} \\
|\nabla A | ^ 2 - \frac  {1 }  {2} |\nabla H| ^ 2 &\geq \frac  {1}{3 } | \nabla A | ^ 2, \label{eqn_gradient2}\\
|\nabla A| ^ 2 & \geq  2 \nabla_{\!\!evol} K^{\perp}.
\label{eqn_gradient3}
\end{align}
\end{lemma}

$ \ $
\section{Preservation of curvature pinching}

In this section, we prove the following curvature pinching condition preserves along the mean curvature flow in the space form $\mathbb{F}^{4}(\bar{K})$.
\begin{proposition}\label{baochi}
Let $F_0 :M \rightarrow \mathbb{S}^{4}(\frac{1}{\sqrt{\bar K}})$ be a compact surface immersed in the sphere with
constant curvature $\bar K$. Suppose $ |A| ^ 2 + 2 \gamma | K^\perp| < k |H|^ 2+\beta\bar K $,
 $\beta=4k-2$, then this condition holds
  along the mean curvature flow for all time $t \in [ 0,T )$ where $T\leq\infty$.
\end{proposition}

\begin{proof}
Suppose the compact surface satisfies $|A| ^ 2 + 2 \gamma | K^\perp| -k |H|^ 2-\beta\bar K<0$ at the initial time.
Denote by
\begin{equation}
  Q(x,t,\bar K)=|A| ^ 2 + 2 \gamma | K^\perp| -k |H|^ 2-\beta\bar K,
\end{equation}
then $Q<0$ also holds at the initial time. If the curvature pinching condition does not preserve, then there is a first point and a time such that $Q(x,t)=0$, where $x$ is the maximal point of $Q$ at time $t$. From Lemma \ref{ee} and Lemma \ref{eeK}, we have
\begin{align}
\frac { \partial }{ \partial t } Q =& \Delta Q - 2 \left( |\nabla A | ^ 2 + 2\gamma \frac{K^\perp}{|K^\perp|}\nabla_{evol} K^\perp - k|\nabla H| ^ 2 \right) + 2 R_1- 2 kR  _2 +2\gamma R_3\notag \\
&
+4\bar K|H|^2-4\bar K|A|^2-4k\bar K|H|^2-8\gamma\bar K| K^\perp| .\notag
\end{align}
From (\ref{eqn_gradient1}), we have
\begin{equation}|\nabla H| ^ 2\leq \frac43|\nabla A|^2\leq \frac32|\nabla A|^2,
\end{equation}
and combine with (\ref{eqn_gradient3}), we have
\begin{equation}\label{Q}
  - 2 \left( |\nabla A | ^ 2 + 2\gamma \frac{K^\perp}{|K^\perp|}\nabla_{evol} K^\perp - k|\nabla H| ^ 2 \right)\leq (-2+2\gamma+\frac83 k)|\nabla A|^2<0.
\end{equation}

If $|H|=0$, the evolution equation of $Q$ and $(\ref{Q})$ implies
\begin{equation}
  \frac { \partial }{ \partial t } Q \leq \Delta Q + 2 R_1 +2\gamma R_3-4\bar K|A|^2-8\gamma\bar K| K^\perp|.
\end{equation}
Since $|A| ^ 2 + 2 \gamma | K^\perp| =\beta\bar K$ at the point $x$ of time $t$, the reaction terms are
\begin{eqnarray}
  &&2 R_1 +2\gamma R_3-4\bar K|A|^2-8\gamma\bar K| K^\perp|  \\
 &=&2 R_1 +2\gamma | K^\perp|(3|A|^{2}-2b^2)-4\bar K|A|^2-8\gamma\bar K| K^\perp| \nonumber\\
   & \leq& 2 R_1 +3|A|^2(\beta\bar{K}-|A|^{2})-4\bar K|A|^2.\nonumber
\end{eqnarray}
Li-Li's inequality \cite{LL} and the definition of $\beta$ imply
\begin{eqnarray}
  &&2 R_1 +2\gamma R_3-4\bar K|A|^2-8\gamma\bar K| K^\perp|\\
  &\leq& 3|A|^{4}+3|A|^2(\beta\bar{K}-|A|^{2})-4\bar K|A|^2 \nonumber \\
   &=& (3\beta-4)\bar K|A|^2 \nonumber\\
   &<&0.\nonumber
\end{eqnarray}
Thus it is a contradiction and the proposition follows in this case.

If $|H|\neq 0$, denote  by $\frac{d Q}{dt}$ the reaction terms of $\frac { \partial }{ \partial t } Q$ and at the supposed point we have
\begin{equation}
  \big(k-\frac{1}{2}\big)|H|^2=|\mathring{A}|^2+2\gamma|K^\perp|-\beta\bar K,
\end{equation}
then

\begin{eqnarray}
  \frac{d Q}{dt} &=& \big(-\frac{1}{k-1/2}+2\big)4a^{2}b^{2}+\big(-\frac{1}{k-1/2}+2\big)\gamma|K^\perp||\mathring{A}_1|^2 \nonumber\\
   && +\big(-\frac{3}{k-1/2}+6\big)\gamma|K^\perp||\mathring{A}_2|^2+\big(-\frac{1}{k-1/2}+2\big)|\mathring{A}_2|^4 \nonumber\\
   && +\big(-\frac{1+2\gamma^2}{k-1/2}+6\big)|K^\perp|^2\nonumber \\
   &&+\bar K \Big(2|\mathring{A}_1|^2\beta+\frac{2|\mathring{A}|^2\beta-|\mathring{A}_1|^2\beta+3\gamma|K^\perp|\beta}{k-1/2}-8|\mathring{A}|^2-16\gamma|K^\perp|\Big)\nonumber\\
   &&-\bar{K}^{2}\big(\frac{\beta^2}{k-1/2}-4\beta\big).\nonumber
\end{eqnarray}
For the sake of simplicity, we can write the above equation as
\begin{eqnarray}
\frac{d Q}{dt} &=&\frac{d Q}{dt}(x,t,0)+\bar K\Big(2|\mathring{A}_1|^2\beta+\frac{2|\mathring{A}|^2\beta-|\mathring{A}_1|^2\beta}{k-1/2}-8|\mathring{A}|^2\Big)\\
&&+\bar K \big (\frac{3\beta}{k-1/2}-16\big)\gamma|K^\perp|\nonumber\\
&&-\bar{K}^{2}\big(\frac{\beta^2}{k-1/2}-4\beta\big).\nonumber
\end{eqnarray}
Using the numerical calculation for the extracted quadratic forms from $\frac{d Q}{dt}(x,t,0)$ as Proposition 4.1 in \cite{BN}, we have $\frac{d Q}{dt}(x,t,0)<0$.

For other terms, i.e., the coefficient of $\bar K$ is
\begin{equation}
  2\big(\frac{\beta}{k-1/2}-4\big)|\mathring{A}_2|^2+\big(\frac{\beta}{k-1/2}+2\beta-8\big)|\mathring{A}_{1}|^2+\big(\frac{3\beta}{k-1/2}-16\big)\gamma|K^\perp|,
\end{equation}
and the coefficient of $\bar {K}^2$ is
\begin{equation}
  -\big(\frac{\beta^2}{k-1/2}-4\beta\big).
\end{equation}
We choose
\begin{equation}
 4\big(k-\frac12\big) \leq \beta \leq \min \{4\big(k-\frac12\big),4-\frac{2}{k},\frac{16}{3}\big(k-\frac12\big) \}=4k-2\ \ for \ \  \frac12<k\leq\frac{3}{4}.
\end{equation}
Then $\frac { \partial }{ \partial t } Q < 0$, which is a contradiction via the maximum principle. Thus, we conclude $Q<0$ is preserved along the mean curvature flow.
\end{proof}

If the equality of the pinching condition holds somewhere on the initial surface, i.e., $|A| ^ 2 + 2 \gamma | K^\perp| = k |H|^ 2+\beta\bar K $, the similar argument in \cite{B} implies that after some short time the surface satisfies $|A| ^ 2 + 2 \gamma | K^\perp| < k |H|^ 2+\beta\bar K$.

\begin{remark}\label{Z0}

(i) When $\bar{K}=0$, the pinching condition in Theorem \ref{C} can be replaced by $|A| ^ 2 + 2 \gamma | K^\perp| < k |H|^ 2.$
(ii) When $\bar{K}<0$, we choose
\begin{equation}
  \beta=\max \{ 4\big(k-\frac12\big),4-\frac{2}{k},\frac{16}{3}\big(k-\frac12\big) \}=4-\frac{2}{k} \ \ for \ \  \frac12<k\leq\frac{3}{4}.
\end{equation}
\end{remark}

$ \ $
\section{Convergence theorem in the sphere}
In this section, we prove the convergence theorem for the mean curvature flow of surfaces in $\mathbb{S}^4(\frac{1}{\sqrt{\bar K}})$. First, we derive an estimate for the traceless second fundamental
form, which guarantees that $M$ becomes
spherical along the mean curvature flow.

\begin{proposition}\label{san1}
There exist constants $C < \infty$ and $\sigma, \delta>0$ both depending only on the initial surface such that for all time $t \in [0,T)$ where $T\leq\infty$, we have the estimate

\begin{align}\label{pinch}
|\mathring A|^ 2 + 2\gamma |K^{\perp}| \leq  C (|H| ^2+\bar{K})^{1 -\sigma}e^{-\delta t}.
\end{align}
\end{proposition}
Here we want to find the upper bound of $f_{\sigma} := ( |\mathring A|^2 + 2\gamma|K^{\perp}| )/ (\alpha|H|^{2}+\beta\bar{K})^{1-\sigma}$  by a Stampacchia iteration procedure as in \cite{B}, where $\alpha=k-1/2$. In view of the Proposition \ref{baochi}, we substitute $\gamma=1-\frac{4}{3}k-\epsilon_\nabla$ as in \cite{BN} to keep the gradient term negative. Now, we need the evolution equation of $f_{\sigma}$.
\begin{lemma}\label{yanhua1}
For every $\sigma\in (0,1)$ and $\epsilon_\nabla=1-\frac{4}{3}k-\gamma$, we have the evolution equation
\begin{eqnarray}\label{inequ}
\frac \partial { \partial t } f _{ \sigma} &\leq& \Delta f _\sigma + \frac{ 2 \alpha( 1- \sigma) }{ \alpha|H|^{2}+\beta\bar{K} } \langle \nabla _i |H|^2, \nabla _ i f _ \sigma \rangle \\
&&- \frac { 2 \epsilon _\nabla } { (\alpha|H|^{2}+\beta\bar{K}) ^ { 1-\sigma} }| \nabla A | ^ 2 + 2 \sigma | A| ^ 2 f _ \sigma-\bar K f_\sigma.\nonumber
\end{eqnarray}
\end{lemma}
\begin{proof}
From Lemma \ref{ee} and Lemma \ref{eeK}, we have
\begin{align}\label{shijian}
\frac{\partial }{\partial t}f_{\sigma} &= \frac{\Delta|{A}|^2 - 2|{\nabla A}|^2 + 2R_1 }{ (\alpha|H|^{2}+\beta\bar{K})^{1-\sigma} }  +   \frac{ 2\gamma \left( \Delta| K^{\perp} | - 2 \frac{K^{\perp}}{| K^{\perp}|} \nabla_{evol} K^{\perp}  + R_3 \right)}{ (\alpha|H|^{2}+\beta\bar{K})^{1-\sigma} }&  \\
        &\quad- \frac{1}{2}\frac{ ( \Delta|H|^2 - 2|\nabla H|^2 + 2R_2 ) }{ ( \alpha|H|^{2}+\beta\bar{K})^{1-\sigma} }& \notag\\
        &\quad- \frac{ \alpha(1-\sigma)(|A|^2 + 2\gamma|K^\perp| - \frac{1}{2}|H|^2) }{ (\alpha|H|^{2}+\beta\bar{K})^{2-\sigma} }( \Delta|H|^2 - 2|\nabla H|^2 + 2R_2)&\notag\\
&\quad -4\bar K (2-\sigma)\frac{|\mathring A|^2+2\gamma|K^\perp|}{(\alpha|H|^{2}+\beta\bar{K})^{1-\sigma}}+4\beta{\bar K}^{2} \frac{(1-\sigma)f_{\sigma}}{\alpha|H|^{2}+\beta\bar{K}}.\notag
\end{align}
By direct computation, we have
\begin{align}\label{daoshu}
\Delta f_{\sigma} &= \frac{ \Delta ( |\mathring A|^2 + 2\gamma|K^{\perp}|) }{ ( \alpha|H|^{2}+\beta\bar{K} )^{1-\sigma} } - \frac{ \alpha(1-\sigma)(|\mathring A|^2 + 2\gamma|K^{\perp}|) }{ ( \alpha|H|^{2}+\beta\bar{K} )^{2-\sigma} } \Delta|{H}|^2\\
        &\quad - \frac{ 2\alpha(1-\sigma) }{ ( \alpha|H|^{2}+\beta\bar{K} )^{2-\sigma} } \big\langle\nabla_i (|\mathring A|^2 + 2\gamma|K^{\perp}|), \nabla_i|{H}|^2 \big\rangle   \notag\\
        &\quad+ \frac{ {\alpha}^{2}(2-\sigma)(1-\sigma)(|\mathring A|^2 + 2\gamma|K^\perp|) }{ ( \alpha|H|^{2}+\beta\bar{K} )^{3-\sigma} } |{\nabla|{H}|^2}|^2,\notag
\end{align}
and
\begin{multline}
-\frac{ 2\alpha(1-\sigma) }{ (\alpha|H|^{2}+\beta\bar{K} )^{2-\sigma} } \big\langle \nabla_i (|\mathring A|^2 +2\gamma|K^\perp|), \nabla_i|{H}|^2 \big\rangle \\ = -\frac{ 2\alpha(1-\sigma) }{ \alpha|H|^{2}+\beta\bar{K} } \big\langle \nabla_i|{H}|^2, \nabla_i f_{\sigma} \big\rangle - \frac{ 2{\alpha}^{2}(1-\sigma)^2 }{ (\alpha|H|^{2}+\beta\bar{K})^2 } f_{\sigma}|{ \nabla|{H}|^2 }|^2.\notag
\end{multline}
From (\ref{shijian})and (\ref{daoshu}), we have
\begin{align}\label{fangcheng}
\frac{\partial}{\partial t} f_{\sigma} &= \Delta f_{\sigma} + \frac{ 2\alpha(1-\sigma) }{ \alpha|H|^{2}+\beta\bar{K} } \big\langle \nabla_i|{H}|^2, \nabla_i f_{\sigma} \big\rangle + \frac{ 2\alpha\sigma R_2 f_{\sigma} }{ \alpha|H|^{2}+\beta\bar{K} } \\
        &\quad - \frac{2\left( |{\nabla A}|^2 + 2\gamma \frac{K^{\perp}\nabla_{evol} K^{\perp}}{|K^{\perp}|} - \alpha\frac{ |\mathring A|^2 +2\gamma|{ K^{\perp} }| }{ \alpha|H|^{2}+\beta\bar{K} }|{ \nabla H }|^2-\frac12|{ \nabla H }|^2 \right)}{(\alpha|H|^{2}+\beta\bar{K})^{1-\sigma}} \notag \\
&\quad-4\bar K(2-\sigma)f_\sigma+4\beta{\bar K}^{2} \frac{(1-\sigma)f_{\sigma}}{\alpha|H|^{2}+\beta\bar{K}} \notag \\
&\quad - \frac{ {\alpha}^{2}\sigma(1-\sigma) }{ (\alpha|H|^{2}+\beta\bar{K})^{2} } f_{\sigma}|{ \nabla|{H}|^2 }|^2 - \frac{ 2\alpha\sigma(|\mathring A|^2  + 2\gamma|{ K^{\perp} }|) }{ (\alpha|H|^{2}+\beta\bar{K})^{2-\sigma} }|{\nabla H}|^2  \notag\\
        &\quad+ \frac{2}{ (\alpha|H|^{2}+\beta\bar{K})^{1-\sigma} }\left( R_1 + \gamma R_3 -(\alpha \frac{ |\mathring A|^2 + 2\gamma|{ K^{\perp} }| }{ \alpha|H|^{2}+\beta\bar{K} }+\frac12 )R_2\right).\notag
\end{align}
We discard the non-positive term on the last line in (\ref{fangcheng}) under the pinching condition as the proof in Proposition \ref{baochi}, and the gradient terms on the second line satisfies
\begin{multline}
    -2\left( |{\nabla A}|^2 + 2\gamma \frac{K^{\perp}\nabla_{evol} K^{\perp}}{|K^{\perp}|} - \alpha\frac{ |\mathring A|^2 +2\gamma|{ K^{\perp} }| }{ \alpha|H|^{2}+\beta\bar{K} }|{ \nabla H }|^2-\frac12|{ \nabla H }|^2 \right)  \\ \leq -2 ( |{ \nabla A }|^2 - \gamma|{ \nabla A }|^2-\alpha |\nabla H|^{2}-\frac12|{ \nabla H }|^2)
    \leq -  2\epsilon_{\nabla} |{ \nabla A }|^2.
\end{multline}
We complete the lemma with the following inequality.
\begin{multline}
  -4\bar K(2-\sigma)f_\sigma+4\beta{\bar K}^{2} \frac{(1-\sigma)f_{\sigma}}{\alpha|H|^{2}+\beta\bar{K}} \\
   \leq -4\bar K(2-\sigma)f_\sigma+4\bar K(1-\sigma)f_\sigma=-\bar Kf_\sigma.
\end{multline}
\end{proof}

Since the absolute term $\sigma |A|^2 f_\sigma$ in the evolution equation is positive, we cannot use the ordinary maximum principle. As in \cite{BN,H2}, we need the negative gradient terms to proceed the iteration. Contracting the Simons identity \cite{Si} with the second fundamental form $A$ we get
\begin{align}
\frac 12 \Delta | A | ^ 2 & =  \langle A , \nabla^2 H\rangle + | \nabla A |^2+ Z+2\bar K|\mathring A|^2, \label{Simons}
\end{align}
where
\begin{align*}
Z = \sum_{ i,j,p,\alpha,\beta } H_\alpha h _{ ip\alpha } h _{ ij \beta } h _{ pj \beta }  - \sum_{ \alpha,\beta } \bigg( \sum_{i,j} h_{ ij \alpha } h _{ ji \beta} \bigg)^2 - |{Rm}^{\perp}| ^ 2.
\end{align*}

The following lemma is the identity of $\Delta|A|^2$ for surfaces in the space form.
\begin{lemma}\label{Sim}Let $F_0 :M \rightarrow \mathbb{F}^{2+d}(\bar K)$ be a compact surface immersed in the space form. Then
\begin{equation}
 \frac12\Delta|A|^2= \langle A_ , \nabla^2 H\rangle +|\nabla A|^2+ 2K|\mathring A|^{2} -\sum_{ \alpha,\beta }{|R^\perp_{12 \alpha\beta}|}^2,
\end{equation}
where $K$ is the section curvature.

In particular when $d=2$, $$Z+2\bar K|\mathring A|^2=2K|\mathring A|^{2} -2{|K^{\perp}|}^2,$$
\end{lemma}
\begin{proof}
When $n=2$, $R_{ijkl}=K(g_{ik}g_{jl}-g_{il}g_{jk})$, where the section curvature $$K=\bar K+\frac{|H|^2}{4}-\frac{|A^{\circ}|^2}{2}.$$ Then
\begin{eqnarray}\label{Simons}
  \Delta h_{ij\alpha} &=& h_{ijkk\alpha}=h_{kijk\alpha} \\
   &=& h_{kikj\alpha}+h_{li\alpha}R_{lkjk}+h_{kl\alpha}R_{lijk}-h_{ki\beta}R_{jk\alpha\beta} \nonumber\\
   &=& h_{kkij\alpha} +2Kh_{ij\alpha}-KH_{\alpha}g_{ij}-h_{ki\beta}R_{jk\alpha\beta}. \nonumber
\end{eqnarray}
Multiplying (\ref{Simons}) by $h_{ij\alpha}$,
\begin{eqnarray}
 \frac12\Delta|A|^2&=&\langle \Delta A , A \rangle+|\nabla A|^2 \\
 &=& \langle A , \nabla^2 H\rangle +|\nabla A|^2+ 2K|\mathring A|^{2} -\sum_{ \alpha,\beta }{|R^\perp_{12 \alpha\beta}|}^2.\nonumber
\end{eqnarray}
\end{proof}

\begin{lemma}\label{Z2}Let $F_0 :M \rightarrow \mathbb{S}^{4}(\frac{1}{\sqrt{\bar K}})$ be a compact surface immersed in the sphere with
constant curvature $\bar K$. If $F$ satisfies pinching condition (\ref{cons}), then there exists a strictly positive constant $\epsilon_1$ such that
\begin{equation}\label{ZZ}
  2K|\mathring A|^{2} -2|K^{\perp}|^2\geq
           \epsilon_1(|\mathring A|^2+2\gamma|K^\perp|)(\alpha|H|^2+\beta \bar{K}).
\end{equation}
\end{lemma}
\begin{proof}
From (\ref{Ka}) we have
\begin{equation}
  |K^\perp|^2\leq \frac{1}{4}|\mathring A|^4\leq\frac{1-k}{4k-2}|\mathring A|^4, \ for \ \frac12<k\leq\frac34.
\end{equation}
Thus
\begin{equation}
  2K|\mathring A|^{2} -2|K^{\perp}|^2\geq2|\mathring A|^2(K-\frac{1-k}{4k-2}|\mathring A|^2).
\end{equation}
The pinching condition (\ref{cons}) implies
\begin{eqnarray}
  K-\frac{1-k}{4k-2}|\mathring A|^2 &\geq& \bar K+\frac{1}{2}(\frac12|H|^2-|\mathring A|^2)-\frac{1-k}{4k-2}|\mathring A|^2 \\
   &\geq& \frac{3k-1-2k^2}{4(2k-1)}|H|^2+(1-\frac{k}{4k-2}\beta)\bar K \nonumber\\
  &\geq&  \frac{1-k}{4}|H|^2 +(1-\frac{k}{4k-2}\beta)\bar K \nonumber\\
&\geq& (1-k)(\alpha|H|^2+\beta \bar{K}).\nonumber
\end{eqnarray}
\end{proof}

\begin{remark}\label{Zh}When $\bar{K}<0$, under the pinching condition (\ref{conh}) we have
\begin{equation}
  2K|\mathring A|^{2} -2|K^{\perp}|^2\geq
            \epsilon_2(|\mathring A|^2+2\gamma|K^\perp|)|H|^2.\end{equation}
\end{remark}
In fact, the pinching condition implies,
\begin{eqnarray*}
  K-\frac{1-k}{4k-2}|\mathring A|^2
   &\geq& \frac{3k-1-2k^2}{4(2k-1)}|H|^2+(1-\frac{k}{4k-2}\beta)\bar K \\
  &\geq&  \frac{1-k}{4}|H|^2.
\end{eqnarray*}

Now we can get the required Poincar\'{e} inequality.
\begin{lemma}\label{Poincare1}
For every $p\geq 2$ and $\eta>0$ we have the estimate
\begin{eqnarray}
&&\int_{M_t} f ^ p _ \sigma (\alpha|H|^2+\beta \bar{K}) d \mu _t  \\
&&\leq \frac { 3p\eta + 12 }{ \epsilon_1 } \int_{M_t} \frac { f ^ { p-1 }_ \sigma}{ (\alpha| H| ^ 2 +\beta\bar{K})^{1- \sigma }} |\nabla A | ^ 2 d \mu _ t  + \frac{2(p-1)}{\epsilon_1\eta}\int_{M_t} f_{\sigma}^{p-2}|{\nabla f_{\sigma}}|^2 \, d\mu_t.\nonumber
\end{eqnarray}
\end{lemma}
\begin{proof}
From Lemma \ref{Sim}, we have
\begin{align}\label{f1}
        \Delta f_{\sigma} &= \frac{2\big\langle \mathring A_{ij}, \nabla_i\nabla_j H \big\rangle}{(\alpha|H|^{2}+\beta\bar{K})^{1-\sigma}} + \frac{ 2|\nabla \mathring A |^2  }{ (\alpha|H|^{2}+\beta\bar{K})^{1-\sigma}} +  \frac{2(2K|\mathring A|^{2} -2|K^{\perp}|^2 )}{(\alpha|H|^{2}+\beta\bar{K})^{1-\sigma}} \\
        &\quad  - \frac{2\alpha(1-\sigma)}{\alpha|H|^2+\beta \bar{K}} \langle \nabla_i |H|^2, \nabla_i f_{\sigma} \rangle + \frac{{\alpha}^{2}\sigma(1-\sigma)}{(\alpha|H|^{2}+\beta\bar{K})^2} f_{\sigma}|\nabla |H|^2|^2 \nonumber \\
        &\quad + \frac{ 2\gamma \Delta|K^{\perp}| }{ (\alpha|H|^{2}+\beta\bar{K})^{1-\sigma} }-\frac{ \alpha(1-\sigma)(| \mathring A|^2+2\gamma |K^{\perp}|)}{(\alpha|H|^{2}+\beta\bar{K})^{2-\sigma}}\Delta |H|^2.\nonumber
    \end{align}

\begin{align}\label{f11}
        \Delta f_{\sigma} &\geq \frac{2\big\langle \mathring A_{ij}, \nabla_i\nabla_j H \big\rangle}{(\alpha|H|^{2}+\beta\bar{K})^{1-\sigma}}  +  \frac{2(2K|\mathring A|^{2} -2|K^{\perp}|^2 )}{(\alpha|H|^{2}+\beta\bar{K})^{1-\sigma}} \\
        &\quad  - \frac{2\alpha(1-\sigma)}{\alpha|H|^2+\beta \bar{K}} \langle \nabla_i |H|^2, \nabla_i f_{\sigma} \rangle -\frac{ 2\alpha(1-\sigma)f_{\sigma}}{\alpha|H|^{2}+\beta\bar{K}}(|H|\Delta |H|) \nonumber \\
        &\quad + \frac{ 2\gamma \Delta|K^{\perp}| }{ (\alpha|H|^{2}+\beta\bar{K})^{1-\sigma} }+\frac{ 2(\frac14-\alpha\frac{ |\mathring A |^2 +2\gamma |K^{\perp}| }{\alpha|H|^{2}+\beta\bar{K}})|\nabla H |^2  }{ (\alpha|H|^{2}+\beta\bar{K})^{1-\sigma}}\nonumber \\
\quad &\geq \frac{2\big\langle \mathring A_{ij}, \nabla_i\nabla_j H \big\rangle}{(\alpha|H|^{2}+\beta\bar{K})^{1-\sigma}}  +  \frac{2(2K|\mathring A|^{2} -2|K^{\perp}|^2)}{(\alpha|H|^{2}+\beta\bar{K})^{1-\sigma}}-\frac{ 2\alpha(1-\sigma)f_{\sigma}}{\alpha|H|^{2}+\beta\bar{K}}(|H|\Delta |H|) \nonumber\\
&\quad- \frac{2\alpha(1-\sigma)}{\alpha|H|^2+\beta \bar{K}} \langle \nabla_i |H|^2, \nabla_i f_{\sigma} \rangle + \frac{ 2\gamma \Delta|K^{\perp}| }{ (\alpha|H|^{2}+\beta\bar{K})^{1-\sigma} }.\nonumber
    \end{align}
We multiply the equation (\ref{f11}) by $f_\sigma^{p-1}$ and integrate on both sides as Proposition 5.5 in \cite{B}, then
\begin{eqnarray}\label{p1}
  &&\int_{M_t} \frac{2f_{\sigma}^{p-1}(2K|\mathring A|^{2} -2|K^{\perp}|^2 )}{(\alpha|H|^{2}+\beta\bar{K})^{1-\sigma}} d\mu_t  \\
   &\leq& (3p\eta + 10 ) \int_{M_t} \frac { f ^ { p-1 }_ \sigma}{ (\alpha| H| ^ 2 +\beta\bar{K})^{1- \sigma }} |\nabla H | ^ 2 d \mu _ t  + \frac{3(p-1)}{\eta}\int_{M_t} f_{\sigma}^{p-2}|{\nabla f_{\sigma}}|^2d\mu_t \nonumber \\
   &&  -\int_{M_t}\frac { 2\gamma \Delta|K^{\perp}| }{ (\alpha|H|^{2}+\beta\bar{K})^{1-\sigma} }d \mu_ t. \nonumber
\end{eqnarray}
From Lemma \ref{Z2}, we have
\begin{eqnarray*}
  \int_{M_t} \frac{2f_{\sigma}^{p-1}(2K|\mathring A|^{2} -2|K^{\perp}|^2)}{(\alpha|H|^{2}+\beta\bar{K})^{1-\sigma}} d\mu_t &\geq& \int_{M_t} \frac{2f_{\sigma}^{p-1}(\alpha H^2+\beta \bar{K})(|\mathring A|^2+2\gamma|K^\perp|)\epsilon_1}{(\alpha|H|^{2}+\beta\bar{K})^{1-\sigma}} d\mu_t \\
   &=& \int_{M_t} 2f_{\sigma}^{p}(\alpha H^2+\beta \bar{K})\epsilon_1 d\mu_t.
\end{eqnarray*}
Next we just deal with the last term of (\ref{p1}) with the normal curvature.
\begin{eqnarray}
   && \ \ \int_{M_t}- \frac{ 2f_{\sigma}^{p-1} \Delta |K^{\perp}| }{ (\alpha|H|^2+\beta\bar{K})^{1-\sigma} } \, d\mu_t \\
   &=&  \int_{M_t} 2\nabla_i \left( \frac{ f^{p-1} }{ (\alpha|H|^2+\beta\bar{K})^{1-\sigma} } \right) \nabla_i |K^{\perp}| \, d\mu_t  \nonumber\\
   &=&  2(p-1)\int_{M_t} \frac{ f_{\sigma}^{p-2} \nabla_i f_{\sigma} \nabla_i |K^{\perp}| }{ (\alpha|H|^2+\beta\bar{K})^{1-\sigma} } \, d\mu_t - 4\alpha(1-\sigma)\int_{M_t} \frac{ f_{\sigma}^{p-1} |H|\nabla_i |H| \nabla_i |K^{\perp}| }{ (\alpha|H|^2+\beta\bar{K})^{2-\sigma} } \, d\mu_t. \nonumber
\end{eqnarray}
From (\ref{Kac}) and the Cauchy-Schwarz inequality, we have
\begin{eqnarray}
  &&-\int_{M_t} \frac{ 2f^{p-1} \Delta |K^{\perp}| }{(\alpha|H|^2+\beta\bar{K})^{1-\sigma} } d\mu_t  \\
   &\leq& \frac{ (p-1) }{ \eta } \int_{M_t}  f_{\sigma}^{p-2} | \nabla f_{\sigma} |^2 \, d\mu_t +  (p-1)\eta  \int_{M_t}  \frac{ f_{\sigma}^{p-1} | \nabla A |^2 }{ (\alpha|H|^2+\beta\bar{K})^{1-\sigma} } \, d\mu_t \nonumber\\
   && +16\sqrt{\frac{4\alpha}{3}}\int_{M_t} \frac{ f_{\sigma}^{p-1} \sqrt{\alpha}|H||\mathring A|| \nabla A |^2 }{ (\alpha|H|^2+\beta\bar{K})^{2-\sigma} } \, d\mu_t   \nonumber\\
 &\leq& \frac{ (p-1) }{ \eta } \int_{M_t}  f_{\sigma}^{p-2} | \nabla f_{\sigma} |^2 \, d\mu_t +  \big((p-1)\eta+10\big)  \int_{M_t}  \frac{ f_{\sigma}^{p-1} | \nabla A |^2 }{ (\alpha|H|^2+\beta\bar{K})^{1-\sigma} } \ d\mu_t. \nonumber
\end{eqnarray}
The lemma follows by combining these estimates together.
\end{proof}

Now we show that the $L^p$-norm of $f_\sigma$ is bounded for sufficiently large $p$.
\begin{lemma}\label{lem1}For any $p\geq\frac{12}{\epsilon_\nabla}+1$ and
$\sigma\leq \frac{\epsilon_1\sqrt{\epsilon_\nabla}}{8\sqrt{3p}}$,
there exist a constant $C$ depending only on the initial surface such that for all
$t\in [0, T)$ where $T\leq\infty$, we have
\begin{eqnarray}\label{8-ineq}
\bigg(\int_{M_t}f_\sigma^pd\mu_t\bigg)^{\frac{1}{p}}\leq Ce^{-\bar{K}t}.
\end{eqnarray}
\end{lemma}

\begin{proof}For $t\geq t_0$, form Lemma \ref{yanhua1}, we have
\begin{eqnarray}\label{9-ineq1}
  \frac{\partial}{\partial t}\int_{M_t}f_\sigma^pd\mu_t &\leq& \int_{M_t}pf_\sigma^{p-1}\frac{\partial}{\partial t}f_\sigma
d\mu_t \\
   &\leq& -p(p-1)\int_{M_t}f_\sigma^{p-2}|\nabla
f_\sigma|^2d\mu_t \nonumber \\
   && +4(1-\sigma)p\int_{M_t}\frac{\alpha f_\sigma^{p-1}}{\alpha|H|^2+\beta \bar{K}
}|H||\nabla|H|||\nabla f_\sigma| d\mu_t \nonumber\\
   && -2p\epsilon_\nabla
\int_{M_t}\frac{f_{\sigma}^{p-1}}{(\alpha|H|^2+\beta \bar{K})^{1-\sigma}}|\nabla
A|^2d\mu_t \nonumber\\
   && +2\sigma p
\int_{M_t}|A|^2f_\sigma^pd\mu_t-p\bar{K}\int_{M_t}f_\sigma^{p}d\mu_t. \nonumber
\end{eqnarray}
In view of the pinching condition we can estimate
\begin{eqnarray}\label{10-ineq1}
   && 4(1-\sigma)p\int_{M_t}\frac{\alpha f_\sigma^{p-1}}{\alpha |H|^2+\beta \bar{K}
}|H||\nabla|H|||\nabla f_\sigma| d\mu_t \\
   &\leq& \frac{2p}{\mu}\int_{M_t}f_\sigma^{p-2}|\nabla
f_\sigma|^2d\mu_t+3p\mu\int_{M_t}\frac{f_\sigma^{p-1}}{(\alpha|H|^2+\beta \bar{K})^{1-\sigma}}|\nabla A|^2d\mu_t. \nonumber
\end{eqnarray}
Substituting (\ref{10-ineq1}) to (\ref{9-ineq1}),  letting
$\mu=\frac{4}{p-1}$ and
$p\geq\frac{12}{\epsilon_\nabla}+1$  we obtain
\begin{eqnarray}
  \frac{\partial}{\partial t}\int_{M_t}f_\sigma^pd\mu_t  &\leq&
-\frac{p(p-1)}{2}\int_{M_t}f_\sigma^{p-2}|\nabla
f_\sigma|^2d\mu_t \\
   &&  -p\epsilon_\nabla
\int_{M_t}\frac{f_{\sigma}^{p-1}}{|H|^{2(1-\sigma)}}|\nabla
A|^2d\mu_t\nonumber \\
   &&  +2\sigma p
\int_{M_t}(\alpha|H|^2+\beta \bar{K})f_\sigma^pd\mu_t-p\bar{K}\int_{M_t}f_\sigma^{p}d\mu_t. \nonumber
\end{eqnarray}
This together with Lemma \ref{Poincare1} implies
\begin{eqnarray}\label{11-ineq1}
  &&\frac{\partial}{\partial t}\int_{M_t}f_\sigma^pd\mu_t \\
  &\leq& -p(p-1)\bigg(\frac{1}{2}-\frac{4\sigma
}{\eta \epsilon_1
}\bigg)\int_{M_t}f_\sigma^{p-2}|\nabla
f_\sigma|^2d\mu_t \nonumber\\
   && -\bigg(p\epsilon_\nabla-\frac{2\sigma p(3p\eta+12)}{\epsilon_1}\bigg)
\int_{M_t}\frac{f_{\sigma}^{p-1}}{(\alpha|H|+\beta \bar{K})^{1-\sigma}}|\nabla
A|^2d\mu_t \nonumber\\
   && -p\bar{K}\int_{M_t}f_\sigma^{p}d\mu_t. \nonumber
\end{eqnarray}
Now we pick
$\eta=\frac{8\sigma}{\epsilon_1}$ and let $\sigma\leq
\min\big\{\frac{\epsilon_1\sqrt{\epsilon_\nabla}}{8\sqrt{3p}}, \frac{\epsilon_1\epsilon_\nabla}{32}
\big\}=\frac{\epsilon_1\sqrt{\epsilon_\nabla}}{8\sqrt{3p}}$ such that
\begin{equation}
 2\sigma (3p\eta+12) \leq \epsilon_\nabla\epsilon_1=\frac14\epsilon_\nabla\epsilon_1+\frac34\epsilon_\nabla\epsilon_1.
\end{equation}
Then (\ref{11-ineq1}) reduces to
\begin{eqnarray*}
\frac{\partial}{\partial t}\int_{M_t}f_\sigma^pd\mu_t
\leq-p\bar{K}\int_{M_t}f_\sigma^{p}d\mu_t,
\end{eqnarray*}
and this implies
\begin{eqnarray}
\int_{M_t}f_\sigma^pd\mu_t\leq
e^{-p\bar{K}t}\int_{M_{t_0}}f_\sigma^{p}d\mu_t.
\end{eqnarray}
\end{proof}

Then we can proceed by a Stampacchia iteration procedure as in \cite{H1} to bound $f_\sigma$ in $L^\infty$ and complete the proof of Proposition \ref{san1}.

Here we need the following gradient estimation to compare the mean curvature at different points.

\begin{proposition}[\cite{B} Theorem 5.8]
For every $\eta>0$, there exists a constant $C_\eta$ depending only on $\eta$ such that for all time, there holds
$$|\nabla H|^2\leq (\eta |H|^4+C_\eta)e^{-\delta t/2}.$$
\end{proposition}

Now we can complete the proof of the theorem \ref {main2}.
\begin{proof}[proof of theorem \ref {main2}]
From the proof of lemma \ref{Z2} we know $K\geq \epsilon|H|^2$. Combing the gradient estimate we have $\frac{|H|_{min}}{|H|_{max}}\rightarrow 1$ and the diameter of $M_t$ is bounded. By similar arguments as in \cite{B},  we obtain the mean curvature flow with initial value $F$
either converges to a round point in finite time, or
converges to a total geodesic sphere of $\mathbb{S}^{4}(\frac{1}{\sqrt{\bar K}})$ as
$t\rightarrow\infty$.
\end{proof}

$ \ $
\section{Convergence theorem in the hyperbolic space}
In this section we prove the convergence theorem for surfaces that move along the mean curvature flow in $\mathbb{H}^4(\bar K)$. First we set $\tilde{f}_{\sigma} := ( |\mathring A|^2 + 2\gamma|K^{\perp}| )/ |H|^{2(1-\sigma)}$ and get the following proposition for the upper bound of $\tilde{f}_{\sigma}$.

\begin{proposition}\label{san2}
There exist constants $C < \infty$ and $\delta>0$ both depending only on the initial surface such that for all time $t \in [0,T)$ where $T<\infty$, we have the estimate

\begin{align}\label{pinch}
|\mathring A|^ 2 + 2\gamma |K^{\perp}| \leq  C |H| ^{ 2 -\delta}.
\end{align}
\end{proposition}

The proof of Proposition \ref{san2} is similar as Proposition \ref{san1}. First we need the following two lemmas.
\begin{lemma}\label{yanhua}
For every $\sigma\in (0,1)$ and $\epsilon_\nabla=1-\frac{4}{3}k-\gamma$, we have the evolution equation
\begin{eqnarray}\label{inequ}
  \frac \partial { \partial t } \tilde{f} _{ \sigma} &\leq& \Delta \tilde{f} _\sigma + \frac{ 2 ( 1- \sigma) }{ |H|^2 } \langle \nabla _i |H|^2, \nabla _ i \tilde{f} _ \sigma \rangle - \frac { 2 \epsilon _\nabla } { |H| ^ { 2 ( 1-\sigma)} }| \nabla A | ^ 2 \\
   && + 2 \sigma | A| ^ 2 \tilde{f} _ \sigma-4\bar K(2-\sigma)\tilde{f}_\sigma.\nonumber
\end{eqnarray}
\end{lemma}

\begin{lemma}\label{Poincare}
For every $p\geq 2$ and $\eta>0$ we have the estimate
\begin{eqnarray}
  \int_{M_t} \tilde{f} ^ p _ \sigma | H| ^ 2 d \mu _t  &\leq& \frac { 3p\eta + 12}{ \epsilon_2 } \int_{M_t} \frac { \tilde{f} ^ { p-1 }_ \sigma}{ | H| ^ { 2 ( 1- \sigma )}} |\nabla A | ^ 2 d \mu _ t \\
   &&  + \frac{2(p-1)}{\epsilon_2\eta}\int_{M_t} \tilde{f}_{\sigma}^{p-2}|{\nabla \tilde{f}_{\sigma}}|^2 \, d\mu_t. \nonumber
\end{eqnarray}
\end{lemma}
\begin{proof}
From Lemma \ref{Sim}, we have
\begin{align*}
        \Delta \tilde{f}_{\sigma} &= \frac{2\langle \mathring A_{ij}, \nabla_i\nabla_j H \rangle}{|{H}|^{2(1-\sigma)}} + \frac{ 2 |\nabla \mathring A |^2}{ |{H}|^{2(1-\sigma)}}  +  \frac{2(2K|\mathring A|^{2} -2|K^{\perp}|^2 )}{|{H}|^{2(1-\sigma)}} \\
        &\quad  - \frac{2(1-\sigma)}{|H|^2} \langle \nabla_i |H|^2, \nabla_i \tilde{f}_{\sigma} \rangle + \frac{\sigma(1-\sigma)}{(|H|^2)^2} \tilde{f}_{\sigma}|\nabla |H|^2|^2  \\
        &\quad - (1-\sigma)\tilde{f}_{\sigma}\frac{\Delta |H|^2}{|H|^2}+ \frac{ 2\gamma \Delta|K^{\perp}| }{ |H|^{2(1-\sigma) } }.
    \end{align*}
We then multiply by $\tilde{f}_\sigma^{p-1}$ on both sides and integrate. From Remark \ref{Zh}, we have
\begin{eqnarray*}
  \int_{M_t} \frac{2\tilde{f}_{\sigma}^{p-1}}{|{H}|^{2(1-\sigma)}}(2K|\mathring A|^{2} -2|K^{\perp}|^2 ) d\mu_t &\geq& \int_{M_t} \frac{2\tilde{f}_{\sigma}^{p-1}H^2(|\mathring A|^2+2\gamma|K^\perp|)\epsilon_2}{|{H}|^{2(1-\sigma)}} d\mu_t \\
   &\geq& \int_{M_t} 2\tilde{f}_{\sigma}^{p}H^2\epsilon_2 d\mu_t.
\end{eqnarray*}
Then the lemma follows by the proof of Lemma \ref{Poincare1}.
\end{proof}

Now we show that the $L^p$-norm of $\tilde{f}_\sigma$ is bounded for sufficiently large $p$.
\begin{lemma}\label{lem1}For any $p\geq\frac{12}{\epsilon_\nabla}+1$ and
$\sigma\leq \frac{\epsilon_2\sqrt{\epsilon_\nabla}}{8\sqrt{3p}}$,
there exist a constant $C$ depending only on the initial surface such that for all
$t\in [0,T)$ where $T<\infty$, we have
\begin{eqnarray}\label{8-ineq}
\bigg(\int_{M_t}\tilde{f}_\sigma^pd\mu_t\bigg)^{\frac{1}{p}}\leq C.
\end{eqnarray}
\end{lemma}

\begin{proof}For $t\geq t_0$, form Lemma \ref{yanhua}, we have
\begin{eqnarray}\label{9-ineq}
  \frac{\partial}{\partial t}\int_{M_t}\tilde{f}_\sigma^pd\mu_t&\leq& \int_{M_t}p\tilde{f}_\sigma^{p-1}\frac{\partial}{\partial t}\tilde{f}_\sigma
d\mu_t \\
  &\leq&\ -p(p-1)\int_{M_t}\tilde{f}_\sigma^{p-2}|\nabla
\tilde{f}_\sigma|^2d\mu_t \nonumber\\
   && +4(1-\sigma)p\int_{M_t}\frac{\tilde{f}_\sigma^{p-1}}{|H|^2
}|H||\nabla|H|||\nabla \tilde{f}_\sigma| d\mu_t \nonumber\\
   && -2p\epsilon_\nabla
\int_{M_t}\frac{\tilde{f}_{\sigma}^{p-1}}{|H|^{2(1-\sigma)}}|\nabla
A|^2d\mu_t \nonumber\\
   && +2\sigma p
\int_{M_t}|A|^2\tilde{f}_\sigma^pd\mu_t-4(2-\sigma)\bar{K}p\int_{M_t}\tilde{f}_\sigma^{p}d\mu_t. \nonumber
\end{eqnarray}
we have
\begin{eqnarray}\label{10-ineq}
   && 4(1-\sigma)p\int_{M_t}\frac{\tilde{f}_\sigma^{p-1}}{|H|^2
}|H||\nabla|H|||\nabla \tilde{f}_\sigma| d\mu_t \\
   &\leq& \frac{2p}{\mu}\int_{M_t}\tilde{f}_\sigma^{p-2}|\nabla
\tilde{f}_\sigma|^2d\mu_t+3p\mu\int_{M_t}\frac{\tilde{f}_\sigma^{p-1}}{|H|^{2(1-\sigma)}}|\nabla A|^2d\mu_t. \nonumber
\end{eqnarray}
Substituting (\ref{10-ineq}) to (\ref{9-ineq}),  letting
$\mu=\frac{4}{(p-1)}$ and
$p\geq\frac{12}{\epsilon_\nabla}+1$  we obtain
\begin{equation*}
\begin{split}
\frac{\partial}{\partial t}\int_{M_t}\tilde{f}_\sigma^pd\mu_t \leq&\
-\frac{p(p-1)}{2}\int_{M_t}\tilde{f}_\sigma^{p-2}|\nabla
\tilde{f}_\sigma|^2d\mu_t\\
&\ -p\epsilon_\nabla
\int_{M_t}\frac{\tilde{f}_{\sigma}^{p-1}}{|H|^{2(1-\sigma)}}|\nabla
A|^2d\mu_t\\
&\ +2\sigma p
\int_{M_t}|H|^2\tilde{f}_\sigma^pd\mu_t-8\bar{K}p\int_{M_t}\tilde{f}_\sigma^{p}d\mu_t.
\end{split}
\end{equation*}
This together with Lemma \ref{Poincare} implies
\begin{eqnarray}\label{11-ineq}
  \frac{\partial}{\partial t}\int_{M_t}\tilde{f}_\sigma^pd\mu_t &\leq& -p(p-1)\bigg(\frac{1}{2}-\frac{4\sigma
}{\eta \epsilon_2
}\bigg)\int_{M_t}\tilde{f}_\sigma^{p-2}|\nabla
\tilde{f}_\sigma|^2d\mu_t \\
   &&  -\bigg(p\epsilon_\nabla-\frac{2\sigma p(3p\eta+12)}{\epsilon_2}\bigg)
\int_{M_t}\frac{\tilde{f}_{\sigma}^{p-1}}{|H|^{2(1-\sigma)}}|\nabla
A|^2d\mu_t \nonumber\\
   && -8\bar{K}p\int_{M_t}\tilde{f}_\sigma^{p}d\mu_t. \nonumber
\end{eqnarray}
Now we pick
$\eta=\frac{8\sigma}{\epsilon_2}$ and
let $\sigma\leq
\min\big\{\frac{\epsilon_2\epsilon_\nabla}{32}
,\frac{\epsilon_2\sqrt{\epsilon_\nabla}}{8\sqrt{3p}}\big\}$.
Then (\ref{11-ineq}) reduces to
\begin{eqnarray*}
\frac{\partial}{\partial t}\int_{M_t}\tilde{f}_\sigma^pd\mu_t
\leq-8\bar{K}p\int_{M_t}\tilde{f}_\sigma^{p}d\mu_t.
\end{eqnarray*}
This implies
\begin{eqnarray}
\int_{M_t}\tilde{f}_\sigma^pd\mu_t\leq
e^{-8\bar{K}pt}\int_{M_{t_0}}\tilde{f}_\sigma^{p}d\mu_t.
\end{eqnarray}
 For $t\geq t_0$,
from Lemma 4.5 in \cite{LXYZ} we know that $T$ is finite and $\int_{M_t}\tilde{f}_\sigma^pd\mu_t$ is bounded.
\end{proof}

 Then we can proceed by a Stampacchia iteration procedure to bound $\tilde{f}_\sigma$ in $L^\infty$ and complete the proof of Proposition \ref{san2}.
Remember that $T$ is the maximum existence time of mean curvature flow, then we have the following gradient estimation.

\begin{proposition}[\cite{LXYZ} Theorem 5.1]
For every $\eta>0$, there exists a constant $C_\eta$ independent of $t$ such that for all $t\in [0,T)$, there holds
$$|\nabla H|^2\leq \eta |H|^4+C_\eta$$
\end{proposition}

Now we can complete the proof of the theorem \ref {main}.
\begin{proof}[proof of theorem \ref {main}]
First, from Remark \ref{Zh} we know $K\geq \epsilon|H|^2$. Combing the gradient estimate we have $\frac{|H|_{max}}{|H|_{min}}\rightarrow 1$ which is known with the proof in \cite{H1}, so the diameter of $M_t$ tends to zero. By similar rescaling arguments as Theorem 6.1 in \cite{BN},  then we obtain the rescaled  mean curvature flow that converges to an umbilical  surface.
\end{proof}

In particularity, these convergence theorems in space forms also imply the differentiable sphere theorem.
\begin{corollary}
Let $F_0 :M\rightarrow \mathbb{F}^{4}(\bar K)$ be a compact surface immersed in the space form with $|H|^2+4\bar K>0 $. If $M$ satisfies
\begin{equation}|A|^2+2\gamma|K^\perp| \leq k|H|^2+\beta \bar K,
\end{equation}
where $\beta=\begin{cases}
            4k-2, \ & \bar{K}\geq0,\\
            4-\frac{2}{k}, \ & \bar{K}<0,
          \end{cases}$
 then $M$ is diffeomorphic to the unit-sphere.
\end{corollary}

\subsection*{Acknowledgment}
This paper forms Chapter 4 of the first named author's doctoral dissertation \cite{Pu} at Zhejiang University.
He wishes to express his gratitude to Professor Fang Zheng  and Professor Hongwei Xu for their guidance during the preparation.

$$ \ \ \ \ $$

\end{document}